\documentclass[a4paper,12pt]{amsart}
\usepackage{amsmath,amssymb,amscd,amsxtra}
\usepackage{bm}

\theoremstyle{plain}
 \newtheorem{theo}{Theorem}[section]
 \newtheorem{prop}[theo]{Proposition}
 \newtheorem{lemm}[theo]{Lemma}
 \newtheorem{coro}[theo]{Corollary}

\theoremstyle{definition}
 \newtheorem*{defn}{Definition}
 
 \newtheorem{rem}{Remark}[section]

\theoremstyle{remark}
 \newtheorem*{note}{Note}

\newcommand{\field}[1]{\mathbb{#1}}
\newcommand{\C}{\field{C}}

\newcommand{\Q}{\field{Q}}
\newcommand{\R}{\field{R}}
\newcommand{\Z}{\field{Z}}

 \DeclareMathOperator{\Hom}{Hom}
 
\DeclareMathOperator{\DHF}{DH} 
 \DeclareMathOperator{\vol}{vol}

\DeclareMathOperator{\ch}{ch}

\DeclareMathOperator{\Rat}{Rat}
\DeclareMathOperator{\Td}{Td}
\DeclareMathOperator{\Lk}{Lk}

\newcommand{\sigm}{\Sigma}
\newcommand{\sigmn}{\Sigma^{(n)}}
\newcommand{\sigmk}{\Sigma^{(k)}}
\newcommand{\sigmone}{\Sigma^{(1)}}
\newcommand{\del}{\Delta}

\newcommand{\V}{\mathcal{V}}
\newcommand{\T}{\mathcal{T}}

\newcommand{\PO}{\mathcal{P}}
\newcommand{\POxi}{\mathcal{P}(\xi)}

\newcommand{\POxiJ}{\mathcal{P}(\xi)_J}

\newcommand{\barx}{\bar{x}}

\newcommand{\tilG}{\tilde{G}}
\newcommand{\tilT}{\tilde{T}}

\newcommand{\ujJ}{u_j^J}
\newcommand{\uiI}{u_i^I}

\def\l{\langle}
\def\r{\rangle}

\newcommand{\Proj}{\field{P}}

\newcommand{\img}{\sqrt{-1}}

\title{On a Morelli type expression of cohomology classes 
of torus orbifolds}

\author{Akio Hattori}
\address{Graduate School of Mathematical Science, University of Tokyo,
Tokyo, Japan}
\email{hattori@ms.u-tokyo.ac.jp}

\date{}

\begin{document}

\begin{abstract}
Let $X$ be a complete toric variety of dimension 
$n$ and $\del$ the fan in a lattice $N$ associated to 
$X$. For each cone $\sigma$ of 
$\del$ there corresponds an orbit closure $V(\sigma)$ of the 
action of complex torus on $X$. The homology classes 
$\{[V(\sigma)]\mid \dim \sigma=k\}$ form a set of specified 
generators of $H_{n-k}(X,\Q)$. Then any $x\in H_{n-k}(X,\Q)$ can be 
written in the form 
\[ x=\sum_{\sigma\in\del_X, \dim\sigma=k}\mu(x,\sigma)[V(\sigma)]. \]
A question occurs whether there is some canonical way to express 
$\mu(x,\sigma)$. Morelli \cite{Mor} gave an answer when $X$ is 
non-singular and at least for 
$x= \T_{n-k}(X)$ the Todd class of $X$. However his answer takes 
coefficients in the field of rational functions 
of degree $0$ on the Grassmann manifold $G_{n-k+1}(N_\Q)$ of 
$(n-k+1)$-planes in $N_\Q$. His proof uses Baum-Bott's residue 
formula for holomorphic foliations applied to the action of complex 
torus on $X$

On the other hand there appeared several attempts for generalizing 
non-singular toric varieties in topological contexts \cite{DJ, Mas, 
HM1, MP, IFM, CGP}. 
Such generalized manifolds of dimension $2n$ acted on by a compact 
$n$ dimensional torus $T$ are called by the names quasi-toric manifolds, 
torus manifolds, toric manifolds, toric origami manifolds, 
topological toric manifolds and so on. 
Similarly torus orbifold can be considered. 
To a torus orbifold $X$ a simplicial set $\del_X$ called multi-fan of $X$ 
is associated. 
A question occurs whether a similar expression to Morelli's formula holds 
for torus orbifolds.
It will be shown the answer is yes in this case too at least when 
the rational cohomology ring $H^*(X)_\Q$ is generated by $H^2(X)_\Q$. 
Under this assumption the equivariant cohomology ring with rational 
coefficients $H^*_T(X,\Q)$ is isomorphic to $H^*_T(\del_X,\Q)$, the face 
ring of the multi-fan $\del_X$, and the 
proof is carried out on $H^*_T(\del_X,\Q)$ by using completely combinatorial 
terms. 
\end{abstract}

\maketitle

\section{Introduction}

Let $X$ be a complete toric variety of dimension $n$ and $\del_X$ the fan 
associated to $X$. $\del_X$ is a collection of rational convex cones 
in $N_\R=N\otimes \R$ where $N$ is a lattice of rank $n$. 
For each $k$-dimensional cone $\sigma$ in $\del_X$, let 
$V(\sigma)$ be the corresponding orbit closure of dimension 
$n-k$ and $[V(\sigma)]\in A_{n-k}(X)$ be its Chow class. Then the 
Todd class $\T_{n-k}(X)$ of $X$ can be written in the form
\begin{equation}\label{eq:Todd} 
\T_{n-k}(X)=\sum_{\sigma\in\del_X, \dim\sigma=k}\mu_k(\sigma)[V(\sigma)]. 
\end{equation}
However, since the $[V(\sigma)]$ are not linealy independent, the 
coefficients $\mu_k(\sigma)\in\Q$ are not determined uniquely. 
Danilov \cite{Dan} asks if $\mu_k(\sigma)$ can be chosen so that it 
depends only on the cone $\sigma$ not depending on a particular fan 
in which it lies. 

The equality \eqref{eq:Todd} 
has a close connection with the number $\#(P)$ of lattice points  
contained in a convex lattice polytope $P$ in $M_\R$ where $M$ is the 
dual lattice of $N$. For a positive 
integer $\nu$ the number $\#(\nu P)$ is expanded as a polynomial 
in $\nu$ (called Ehrhart polynomial): 
\[ 
\#(\nu P)=\sum_ka_k(P)\nu^{n-k}.
\]

A convex lattice polytope $P$ in $M_\R$ determines a complete toric 
variety $X$ 
and an invariant Cartier divisor $D$ on $X$. There is a one-to-one 
correspondence between the cells $\{\sigma\}$ of $\del_X$ and the 
faces $\{P_\sigma\}$ of $P$. Then the coefficient $a_k(P)$ 
has an expression
\begin{equation}\label{eq:Ehr}
 a_k(P)=\sum_{\dim \sigma=k}\mu_k(\sigma)\vol P_\sigma 
\end{equation}
with the same $\mu_k(\sigma)$ as in \eqref{eq:Todd}. 

Hereafter we shall use notation $H_T^*(\ )_\Q$ 
to mean $H_T^*(\ )\otimes\Q$ and so on. 

We shall restrict ourselves to the case where $X$ is non-singular. 
Put $D_i=[V(\sigma_i)]$ for the one dimensional cone $\sigma_i$, and 
let $x_i\in H^2(X)$ denotes the Poincar\'{e} dual of $D_i$. The divisor 
$D$ is written in the form $D=\sum_id_iD_i$ 
with positive integers $d_i$. Put $\xi=\sum_id_ix_i$. 
It is known that
\begin{equation*}
 a_k(P)=\int_Xe^\xi \T^k(X) \ \ \text{and}\ 
 \vol P_\sigma=\int_Xe^\xi x_\sigma,
\end{equation*}
where $\T^k(X)\in H^{2k}(X)_\Q$ is the $k$-th component of 
the Todd cohomology class, the Poincar\'{e} dual of $\T_{n-k}(X)$, 
and $x_\sigma\in H^{2k}(X)$ is the Poincar\'{e} dual of $[V(\sigma)]$. 
The cohomology class $x_\sigma$ can also be written as 
$x_\sigma=\prod_jx_j$ where the product runs over such $j$ that $\sigma_j$ 
is an edge of $\sigma$. Then the equality \eqref{eq:Ehr} can be rewritten 
as
\begin{equation}\label{eq:Ehr2}
\int_Xe^\xi \T^k(X)=\sum_{\dim\sigma=k}\mu_k(\sigma)\int_Xe^\xi x_\sigma.
\end{equation}
The reader is referred to \cite{Ful} Section 5.3 for details and Note 17 
there for references. 

In his paper \cite{Mor} Morelli gave an answer to Danilov's question. Let 
$\Rat(G_{n-k+1}(N_\Q)))_0$ denote the field of rational functions of 
degree $0$ on the Grassmann manifold of $(n-k+1)$-planes in 
$N_\Q$. For a cone $\sigma$ of dimension $k$ in $N_\R$ 
he associates a rational function $\mu_k(\sigma)\in
\Rat(G_{n-k+1}(N_\Q)))_0$. With this $\mu_k(\sigma)$, the right 
hand side of \eqref{eq:Todd} belongs to 
\[
\Rat(G_{n-k+1}(N_\Q)))_0\otimes_\Q A_{n-k}(X)_\Q,
\]
and the equality \eqref{eq:Todd} means that the rational function 
with values in $A_{n-k}(X)_\Q$ in the right hand side is in fact 
a constant function equal to $\T_{n-k}(X)$ in $A_{n-k}(X)_\Q$. 
In other words this means that 
\[
\sum_{\sigma\in\del_X, \dim\sigma=k}\mu_k(\sigma)(E)[V(\sigma)]
 =\T_{n-k}(X)
\]
for any generic $(n-k+1)$-plane $E$ in $N_\Q$. 

Morelli gives an explicit formula for $\mu_k(\sigma)$ when the toric 
variety is non-singular using Baum-Bott's residue formula for 
singular foliations \cite{BB} applied to the action of $(\C^*)^n$ on $X$. 
He then shows that the function 
$\mu_k(\sigma)$ is additive with respect to non-singular subdivisions 
of the cone $\sigma$. This fact leads to \eqref{eq:Todd} in 
its general form.  

One can ask a similar question about general classes other than the 
Todd class whether it is possible to define 
$\mu(x,\sigma)\in\Rat(G_{n-k+1}(N_\Q)))_0$ for 
$x\in A_{n-k}(X)$ in a canonical way to satisfy 
\begin{equation}\label{eq:Todd'}
x=\sum_{\sigma\in\del_X, \dim\sigma=k}\mu(x,\sigma)[V(\sigma)]. 
\end{equation}
When $X$ is non-singular one can expect that $\mu(x,\sigma)$ satisfies 
a formula analogous to \eqref{eq:Ehr2}
\begin{equation}\label{eq:Ehr'} 
\int_X e^\xi x=\sum_{\dim\sigma=k}\mu(x,\sigma)\int_Xe^\xi x_\sigma
\end{equation}
for any cohomology class $\xi=\sum_id_ix_i$. In this sense the 
formula does not explicitly refer to convex polytopes.  
Fulton and Sturmfels \cite{FS} used Minkowski weights to describe 
intersection theory of toric varieties. For complete 
non-singular varieties or $\Q$-factorial varieties $X$ the Minkowski 
weight $\gamma_x:H^{2(n-k)}(X)\to \Q$ corresponding to $x\in H^{2k}(X)$ 
is defined by $\gamma_x(y)=\int_X xy$. Thus, if the $d_i$ are considered 
as variables in $\xi$, the formula 
\eqref{eq:Ehr'} is considered as describing $\gamma_x$ as 
a linear combination of the Minkowski weights of $\gamma_{x_\sigma}$. 

On the other hand topological analogues of toric variety were discussed by several 
authors \cite{DJ, Mas, HM1, MP, IFM, CGP}. Most general one would be torus 
orbifold \cite{HM1}. To a torus orbifold $X$ a multi-fan $\del_X$ is associated. 
Multi-fan is a generalized notion of fan. Its cohomology reflects the cohomology of the torus orbifold.  

The purpose of the present paper is to establish the 
formula \eqref{eq:Ehr'} by showing an explicit formula for $\mu(x,\sigma)$ 
when $X$ is a torus orbifold. Moreover our proof is 
based on a simple combinatorial argument carried on the associated multi-fan $\del_X$. 
Topologically the formula 
concerns equivariant cohomology classes on torus orbifolds. 
This would suggest that actions of compact tori 
equipped with some nice conditions admit topological residue formulas 
similar to Baum-Bott' formula. 

In Section 2 we recall the definition of multi-fans and torus 
orbifolds together with relevant facts. 
The definition of $\mu(x,\sigma)$ is given for multi-fans and 
consequently for torus orbifolds. Theorem \ref{theo:multimain} states that 
the formula \eqref{eq:Todd} holds for any 
torus orbifolds. Furthermore Corollary \ref{coro:multiak} 
ensures that the formula \eqref{eq:Ehr} holds for torus orbifolds.
Finally Corollary \ref{coro:multimain} states that \eqref{eq:Todd'} 
holds for a torus orbifolds $X$ such that $H^*(X)_\Q$ is 
generated by $H^2(X)_\Q$.

\section{Multi-fans and torus orbifolds}
The notion of multi-fan and multi-polytope were introduced in 
\cite{Mas}. In this article 
we shall be concerned only with simplicial multi-fans. See 
\cite{Mas, HM1, HM2} for details. 

Let $N$ be a lattice of rank $n$. A \emph{simplicial multi-fan} in $N$ 
is a triple $\Delta=(\Sigma,C,w)$ where $\sigm=\bigsqcup_{k=0}^n\sigmk$ 
is an (augmented) simplicial 
complex, $\sigmk$ being the set of $k-1$ simplices, $C$ is a map 
from $\sigmk$ into the set of $k$-dimensional 
strongly convex rational
polyhedral cones in the vector space $N_\R=N\otimes \R$ for 
each $k$, and  $w$ is a map $\sigmn \to \Z$. $\sigm^{(0)}$ consists 
of a single element $o=\text{the empty set}$. (The definition in 
\cite{Mas} and \cite{HM1} requires additional restriction on $w$.) 
We assume that any $J\in \Sigma$ 
is contained in some $I\in \sigmn$ and $\sigmn$ is not empty. 

The map $C$ is required to satisfy the following condition; if 
$J\in\Sigma$ 
is a face of $I\in\Sigma$, then $C(J)$ is a face of $C(I)$, and   
for any $I$, the map $C$ restricted on 
$\sigm(I)=\{J\in \Sigma\mid J\subset I\}$ is an isomorphism of 
ordered sets onto the set of faces of $C(I)$. It follows that 
$C(I)$ is necessarily a simplicial cone and $C(o)=0$.
A simplicial fan is considered as a simplicial multi-fan such that 
the map $C$ on $\sigm$ is injective and $w\equiv 1$. 

For each $K\in \Sigma$ we set
\[ \Sigma_K=\{ J\in\Sigma\mid K\subset J\}.\]
It inherits the partial ordering from $\Sigma$ and becomes a 
simplicial set where $\sigm_K^{(j)}\subset\sigm^{(j+|K|)}$. $K$ is 
the unique element in $\Sigma_K^{(0)}$. 
Let $N_K$ be the minimal primitive sublattice of $N$ containing 
$N\cap C(K)$, and 
$N^K$ the quotient lattice of $N$ by $N_K$. 
For $J\in \Sigma_K$ we define $C_K(J)$ to be
the cone $C(J)$ projected on $N^K\otimes\R$. 
We define a function 
\[ w:\Sigma_K^{(n-|K|)}\subset \Sigma^{(n)} \to \Z \]
to be the restrictions of $w$ to $\Sigma_K^{(n-|K|)}$. The triple 
$\Delta_K=(\Sigma_K,C_K,w)$ is a multi-fan in $N^K$ 
and is called the  
\emph{projected multi-fan} with respect to $K\in \Sigma$. 
For $K=o$, the projected multi-fan $\Delta_o$ 
is nothing but $\Delta$ itself. 

A vector $v\in N_\R$ will be called \emph{generic} if $v$ does
not lie on any linear subspace spanned by a cone in 
$C(\Sigma)$ of dimension less than $n$. For a generic
vector $v$ we set $d_v=\sum_{v\in C(I)}w(I)$, where
the sum is understood to be zero if there are no such $I$. 

\begin{defn}
A simplicial multi-fan $\Delta=(\Sigma,C,w)$ is called 
\emph{pre-complete} if the integer $d_v$ is independent of 
generic vectors $v$. In this case this integer will be called 
the {\it degree} of $\Delta$ and will be denoted by $\deg(\Delta)$.
It is also called the \emph{Todd genus} of $\del$ and is denoted 
by $\Td[\del]$. A pre-complete 
multi-fan $\Delta$ is said to be \emph{complete} if the 
projected multi-fan $\Delta_K$ is pre-complete for every $K\in \Sigma$. 
\end{defn}
A multi-fan is complete if and only if the projected multi-fan 
$\Delta_J$ is pre-complete for every $J\in \Sigma^{(n-1)}$.

Like a toric variety gives rise to a fan, a torus orbifold gives rise 
to a complete simplicial multi-fan, though this correspondence is not 
one to one. A torus orbifold is a \underline{closed oriented} 
orbifold with an effective action (in the sense of orbifold action) 
of a compact torus of half the dimension of 
the orbifold with non-empty fixed point set and with some additional 
conditions on the orientations of certain type of suborbifolds 
(precise statement will be given later. 
See \cite{Sat} for terminologies 
concerning orbifolds, and \cite{HM1} for those of torus orbifolds). 
Cobordism invariants of torus orbifolds are encoded in the 
associated multi-fans.

Let $X$ be a torus orbifold. A connected component of the fix point set of 
a subcircle of the torus $T$ is a suborbifold. A suborbifold of this type 
which has 
codimension two and contains at least one fixed point of the action of 
$T$ is called characteristic suborbifold. By the orientation convention 
included in the definition of torus orbifold, a characteristic suborbifold 
is equipped with a fixed orientation.

In the following, characteristic suborbifolds will be denoted by$X_i$. 
In the multi-fan $\del(X)=(\sigm(X),C(X),w(X))$ the simplicial
complex $\sigm(X)$ is given by 
\[ 
\Sigma^{(k)}(X)=\{ I\mid \#I=k+1,(\bigcap_{i\in I}X_i)^T\not=\emptyset\}.
\]

Let $S_i$ be the circle that fixes the points of $X_i$. Take a point 
$x$ in $X_i$. Take an orbifold chart $(U_x,V_x,H_x,p_x)$ around $x$ in which 
$U_x$ is invariant under the action of $S_i$ and $V_x$ is an Euclidean  
ball on which $H_x$ acts linearly and the projection $p_x:V_x\to U_x$ 
identifies $V_x/H_x$ with $U_x$. Then there exist  
a covering group $\tilde{S}_i$ of $S_i$ and a lifting of the action of $S_i$ 
to the action of $\tilde{S}_i$ on $V_x$ (exactly its tangent space). 
Hereafter we shall always take the minimal covering with the above property.

If $x$ is a fixed point of the action of $T$, $U_x$ can be taken 
invariant under the action of $T$ and such that $p_x^{-1}(x)$ is a single point. 
Furthermore if $x$ is in a characteristic 
suborbifold $X_i$, then the vector space $V_x$ decomposes into a direct
sum $V_x=V_i+V_i^\bot$ where $V_i^\bot$ is tangent to 
$p_x^{-1}(U_x\bigcap X_i)$ and $V_i$ is normal to the tangent space of 
$p_x^{-1}(U_x\bigcap X_i)$ at $p_x^{-1}(x)$ and is endowed with an invariant 
complex $1$-dimensional vector space structure as follows from the 
definition of torus orbifolds. Then there is a unique isomorphism 
$\varphi_i:S^1\to \tilde{S}_i$ such that $\varphi_i(z)$ acts by the 
complex multiplication of $z\in S^1\subset\C$ on $V_i$.
$\varphi_i$ 
depends only on $X_i$, not on particular choice of $x$. 
Let $\pi:\tilde{S}_i\to S_i$ denote the covering projection. The homomorphism  
$\rho_i=\pi\circ\varphi_i: S^1\to S_i\subset T$ defines an element 
$v_i\in \Hom(S^1,T)=H_2(BT,\Z)$. Then $C(X)(I)$ is the cone in $N=H_2(BT,\Z)$ 
with apex at $0$ and spanned by $\{v_i\mid i\in I\}$.

Let $\Delta=(\Sigma,C,w)$ be a simplicial multi-fan in a lattice $N$. 
The Stanley-Reisner ring or the face ring of the simplicial set 
$\sigm$ is denoted by 
$H_T^*(\del)$. It is the quotient ring of the polynomial ring 
$\Z[x_i\mid i\in\sigmone]$ by the ideal generated by 
\[\{x_K=\prod_{i\in K}x_i\mid K\subset \sigmone,\ K\notin\sigm\}\].

When $\del$ is the fan $\del_X$ associated to a torus orbifold $X$, 
$H_T^*(\del_X)_\Q$ can be 
identified with a subring of 
the equivariant cohomology ring $H_T^*(X)_\Q$ of $X$ with 
respect to the 
action of compact torus $T$ acting on $X$ (see \cite{Mas}). 
(Hereafter we shall use notation $H_T^*(\ )_\Q$ 
to mean $H_T^*(\ )\otimes\Q$.) 

In the sequel we shall often consider a set $\V$ consisting 
of non-zero edge vectors $v_i$ for each $i\in \sigmone$ such that 
$v_i\in N\cap C(i)$. We do not require $v_i$ to be primitive. 
This has meaning for torus orbifolds. 
For any $K\in \Sigma$ put $\V_K=\{v_i\}_{i\in K}$. Let $N_{K,\V}$ 
be the sublattice of $N_K$ generated by $\V_K$. The quotient group 
$N_K/N_{K,\V}$ is denoted by $H_{K,\V}$ . 

Let $\V=\{v_i\}_{i\in \sigmone}$ be a set of prescribed 
edge vectors as before. We define a homomorphism 
$M=N^*=H_T^2(pt) \to H_T^2(\Delta)$ by the formula
\begin{equation}\label{eq:module}
 u=\sum_{i\in\sigmone}\l u,v_i\r x_i, 
\end{equation}
This extends to a homomorphism $H_T^*(pt)\to H_T^*(\Delta)$ and 
makes $H_T^*(\Delta)$  a ring over $H_T^*(pt)$ (regarded as embedded in 
$H_T^*(\Delta)$). 

Since this definition depends on the set $\V$, the $H_T^*(pt)$-module 
structure of $H_T^*(\del)$ also depends on $\V$. To emphasize this fact 
we shall use the notation $H_T^*(\del,\V)$. When all the $v_i$ are taken 
primitive, the notation $H_T^*(\del)$ is used.

Fix $I\in\sigmn$ and let $\{u_i^I\}_{i\in I}$ be the basis 
of $N^*=H^2(pt)$ dual to $\{v_i\}_{i\in I}$.  
Define $\iota_I^*: H_T^2(\del)_\Q\to M_\Q=H_T^2(pt)_\Q$ by
\begin{equation}\label{eq:iota}
\iota_I^*\left(\sum_{i\in\sigmone}d_ix_i\right)=\sum_{i\in I}d_i\uiI. 
\end{equation}
$\iota_I^*$ extends to $H_T^*(\del)_\Q\to H_T^*(pt)_\Q$. 
It is an $H_T^*(pt)_\Q$-module map, since 
\[ \iota_I^*(u)=u \ \text{for $u\in H_T^*(pt)_\Q$}. \]
Let $S$ be the multiplicative set in $H_T^*(pt)_\Q$ generated by 
non-zero elements in $H_T^2(pt)_\Q$. The push-forward 
$\pi_*:H_T^*(\Delta)_\Q\to S^{-1}H_T^*(pt)_\Q$ is defined by 
\begin{equation}\label{eq:pi}
 \pi_*(x)=\sum_{I\in\sigmn}\frac{\iota_I^*(x)}
  {|H_I|\prod_{i\in I}u_i^I}. 
\end{equation}
It is an $H_T^*(pt)_\Q$-module map, and 
lowers the degrees by $2n$. 
It is known \cite{HM1} that, if $\Delta$ is a complete simplicial 
multi-fan, then the image of $\pi_*$ lies in $H_T^*(pt)_\Q$. 

Assume that $\Delta$ is complete. Let $p_*:H_T^*(\Delta)_\Q\to \Q$ 
be the composition of 
$\pi_*:H_T^*(\Delta)_\Q\to H_T^*(pt)_\Q$ and 
$H_T^*(pt)_\Q\to H_T^0(pt)_\Q=\Q$. Note that 
$p_*$ induces $\int_\Delta:H^*(\Delta)_\Q\to \Q$ as noted in 
\cite{HM1} where $H^*(\Delta)_\Q$ is the quotient of $H_T^*(\Delta)_\Q$ 
by the ideal generated by $H_T^+(pt)_\Q$. Note that $H^*(\Delta)_\Q$ 
is defined independently of $\V$. If $\bar{x}$ denotes the 
image of $x\in H_T^*(\Delta)_\Q$ in $H^*(\Delta)_\Q$, then 
$\int_\Delta \bar{x}=p_*(x)$. 

If $X$  a torus orbifold such that $\del_X=\del$ 
then $H_T^*(\del)_\Q$ is a subring of $H_T^*(X)_\Q$. From this it follows that 
$p_*$ on $H_T^*(\Delta)_\Q$ is the restriction of $p_*:H_T^*(X)_\Q\to \Q$ 
and $\int_\del$ is the ordinary integral $\int_X$ (see \cite{HM1}). 

Let $K\in \sigmk$ and let $\Delta_K=(\Sigma_K,C_K,w_K)$ be 
the projected multi-fan. The link $\Lk\ K$ of $K$ in $\Sigma$ 
is a simplicial complex consisting of simplices $J$ such that 
$K\cup J\in \Sigma$ and $K\cap J=\emptyset$. 
It will be denoted by $\Sigma'_K$ in the sequel. There is an isomorphism 
from $\Sigma'_K$ to $\Sigma_K$ sending $J\in {\Sigma'}_K$ to 
$J\cup K$.
Let $\V=\{v_i\}_{i\in \sigmone}$ be a set of prescribed 
edge vectors as before. Let $\{u_i^K\}_{i\in K}$ be the basis of 
$N_{K,\V}^*$ dual to $\V_K$. We consider the polynomial 
ring $R_K$ 
generated by $\{x_i\mid i\in K\cup{\Sigma'}_K^{(1)}\}$ 
and the ideal $\mathcal{I}_K$ generated by monomials
$x_J=\prod_{i\in J} x_i$ such 
that $J\notin \sigm(K)\ast\Sigma'_K$ where $\sigm(K)\ast\Sigma'_K$ 
is the join of $\sigm(K)$ and $\Sigma'_K$. We define the equivariant 
cohomology $H_T^*(\Delta_K)$ of
$\Delta_K$ with respect to the torus $T$ as the quotient ring
$R_K/\mathcal{I}_K$. 

If $\V$ is a set of  prescribed edge vectors, $H_T^2(pt)$ is regarded 
as a submodule of $H_T^2(\Delta_K)$ by a formula similar to 
\eqref{eq:module}. This defines
an $H_T^*(pt)$-module structure on $H_T^*(\Delta_K)$ which 
will be denoted by $H_T^*(\Delta_K,\V)$ to specify the dependence 
on $\V$. The projection $H_T^*(\Delta,\V)\to H_T^*(\Delta_K,\V)$ 
is defined by sending $x_i$ to $x_i$ for $i\in K\cup{\Sigma'}_K^{(1)}$ and
putting $x_i=0$ for $i\notin K\cup{\Sigma'}_K^{(1)}$.  
The restriction homomorphism
$\iota_I^*:H_T^*(\Delta_K,\V)_\Q \to H_T^*(pt)_\Q$ for 
$I\in \Sigma_K^{(n-k)}$ and the push-forward 
$\pi_*:H_T^*(\Delta_K,\V)_\Q \to S^{-1}H_T^*(pt)_\Q$ are
also defined in a similar way as before.  

Given   
$\xi=\sum_{i\in K\cup{\Sigma'}_K^{(1)}}d_ix_i\in H_T^2(\Delta_K,\V)_\R,
\ d_i\in \R$, let $A_K^*$ be 
the affine subspace in the space $M_\R$ defined by 
$\langle u,v_i\rangle =d_i$ for $i\in K$. Then we introduce 
a collection $\mathcal{F}_K=\{F_i\mid i\in{\Sigma'}_K^{(1)}\}$
of affine hyperplanes in $A_K^*$ by setting 
\[ F_i= \{u\mid u\in A_K^*,\ \langle u,v_i \rangle =d_i \}. \]
The pair $\PO_K(\xi)=(\Delta_K,\mathcal{F}_K)$ will be called a 
\emph{multi-polytope} associated with $\xi$; see \cite{HM2}. 
In case $K=o\in \sigm^{(0)}$, $\PO_K(\xi)$ is simply denoted by $\POxi$. 

For $\xi=\sum_{i\in\sigmone}d_ix_i$ and $K\in\sigmk$ put 
$\xi_K=\sum_{i\in K\cup{\Sigma'}_K^{(1)}}d_ix_i$ and 
$\POxi_K=\PO_K(\xi_K)$. It will be called the \emph{face} of $\POxi$ 
corresponding to $K$. 

For $I\in \Sigma_K^{(n-k)}$, i.e. 
$I \in \sigmn$ with $I\supset K$, we put 
$u_I=\cap_{i\in I}F_i=\cap_{i\in I\setminus K}F_i\cap A_K^*\in A_K^*$. 
Note that $u_I$ is equal to $\iota_I^*(\xi)$. The dual vector space 
$(N^K_\R)^*$ of $N^K_\R$ is canonically identified with 
the subspace $M_{K\R}$ of 
$M_\R=H_T^2(pt)_\R$. It is parallel to $A_K^*$, and $u_i^I$ lies 
in $M_{K\R}$ for $I\in \Sigma_K^{(n-k)}$ and $i\in I\setminus K$.
A vector $v\in N^K_\R$ is called generic if
$\langle u_i^I,v\rangle \not=0$ for any $I\in \Sigma_K^{(n-k)}$ and
$i\in I\setminus K$. 
The image in $N_\R^K$ of a generic vector in $N_\R$ is generic. 
Take a generic vector $v\in N^K_\R$, and define
\[ (-1)^I:=(-1)^{\#\{j\in I\setminus K\mid \langle u_j^I,v\rangle >0\}}
\quad \text{and}\quad
(u_i^I)^+ :=
      \begin{cases}
      u_i^I & \text{if}\  \langle u_i^I,v\rangle >0\\
     -u_i^I & \text{if}\  \langle u_i^I,v\rangle <0. 
      \end{cases}            \]
for $I\in \Sigma_K^{(n-k)}$ and $i\in I\setminus K$. 
We denote by $C_K^*(I)^+$ the cone in $A_K^*$ spanned by the
$(u_i^I)^+,\ i\in I\setminus K,$ with apex at $u_I$, and by $\phi_I$ its
characteristic function. With these understood, we define a
function $\DHF_{\PO_K(\xi)}$ on $A_K^*\setminus \cup_iF_i$ by
\[ \DHF_{\PO_K(\xi)}=\sum_{I\in \Sigma_K^{(n-k)}}(-1)^Iw(I)\phi_I. \]
As in \cite{HM2} we call this function the \emph{Duistermaat-Heckman 
function} associated with $\PO_K(\xi)$. When $K=o$, $\DHF_{\POxi}$ is 
defined on $M_\R\setminus \cup_iF_i$.

The following theorem is fundamental in the sequel, cf. \cite{HM2} 
Theorem 2.3 and \cite{HM1} Corollary 7.4. 
\begin{theo}\label{theo:DH}
Let $\Delta$ be a complete simplicial multi-fan.
Let $\xi=\sum_{i\in K\cup{\Sigma'_K}^{(1)}}d_ix_i\in H_T^2(\Delta_K,\V)$ 
be as above with all $d_i$ integers and put 
$\xi_+=\sum_i(d_i+\epsilon)x_i$ with $0<\epsilon<1$. Then
\begin{equation}\label{eq:DH}
\sum_{u\in A_K^*\cap M}\DHF_{\PO_{K}(\xi_+)}(u)t^u
 =\sum_{I\in \Sigma_K^{(n-k)}}\frac{w(I)}{|H_{I,\V}|}\sum_{h\in H_{I,\V}}
 \frac{\chi_I(\iota_I^*(\xi),h)t^{\iota_I^*(\xi)}}
 {\prod_{i\in I\setminus K}(1-\chi_I(u_i^I,h)^{-1}t^{-u_i^I})},
\end{equation}
where $\chi_I(u,h)=e^{2\pi\sqrt{-1}\langle u,v(h)\rangle}$ for 
$u\in N_{I,\V}^*$ and $v(h)$ is a lift of $h\in H_{I,\V}$ to $N_{I,\V}$.
\end{theo}
\begin{note}
The left hand side of \eqref{eq:DH} is considered 
as an element in the group ring of $M$ over $\R$  or the character ring 
$R(T)\otimes\R$ considered as the Laurent polynomial ring in 
$t=(t_1,\dots,t_n)$. The equality shows that the right hand side, which is 
a rational function of $t$, belongs to $R(T)\otimes\R$. 
\end{note}

$\xi=\sum_id_ix_i\in H_T^2(\del,\V)$ is called $T$-Cartier if 
$\iota_I^*(\xi)\in M$ for all $I\in\sigmn$. This condition is 
equivalent to $u_I\in M$ for all $I\in \sigmn$. 
In this case $\POxi$ is said lattice multi-polytope. 
If $\xi$ is $T$-Cartier, then $\chi_I(\iota_I^*(\xi),h)\equiv 1$. Hence 
the above formula \eqref{eq:DH} for $\DHF_{\PO_{K}(\xi_{K+})}$ reduces 
in this case to 
\begin{equation}\label{eq:DH2}
\sum_{u\in A_K^*\cap M}\DHF_{\PO_{K}(\xi_{K+})}(u)t^u
 =\sum_{I\in \Sigma_K^{(n-k)}}\frac{w(I)}{|H_{I,\V}|}\sum_{h\in H_{I,\V}}
 \frac{t^{\iota_I^*(\xi_K)}}
 {\prod_{i\in I\setminus K}(1-\chi_I(u_i^I,h)^{-1}t^{-u_i^I})}.
\end{equation}

Let $H_T^{**}(\ )$ denote the completed equivariant cohomology ring. 
The Chern character $\ch$ sends $R(T)\otimes\R$ to 
$H_T^{**}(pt)_\R$ by $\ch(t^u)=e^u$. 
The image of \eqref{eq:DH2} by $\ch$ is given by 
\begin{equation}\label{eq:DH3} 
\sum_{u\in A_K^*\cap M}\DHF_{\PO_{K}(\xi_{K+})}(u)e^u
 =\sum_{I\in \Sigma_K^{(n-k)}}\frac{w(I)}{|H_{I,\V}|}\sum_{h\in H_{I,\V}}
 \frac{e^{\iota_I^*(\xi_K)}}
 {\prod_{i\in I\setminus K}(1-\chi_I(u_i^I,h)^{-1}e^{-u_i^I})}. 
\end{equation}

Assume that $\xi=\sum_id_ix_i\in H_T^2(\del,\V)$ is $T$-Cartier. The number 
$\#(\PO(\xi)_K)$ is defined by 
\begin{equation*}\label{eq:sharp} 
 \#(\PO(\xi)_K)=\sum_{u\in A_K^*\cap M}\DHF_{\PO_{K}(\xi_{K+})}(u). 
\end{equation*}
It is obtained from \eqref{eq:DH3} by setting $u=0$, that is, it is 
equal to the image of \eqref{eq:DH3} by 
$H_T^{**}(pt)_\Q\to H_T^0(pt)_\Q$. 

The equivariant Todd class $\T_T(\del,\V)$ is defined in such a way
that 
\[ 
\pi_*(e^\xi\T_T(\del,\V))=\sum_{u\in M}\DHF_{\PO(\xi_+)}(u)e^u 
\]
for $\xi$ $T$-Cartier. In order to give the definition we need some 
notations.

For simplicity 
identify the set $\sigmone$ with $\{1,2,\ldots ,m\}$ and  
consider a homomorphism $\eta:\R^m=\R^{\sigmone} \to N_\R$ sending 
$\mathbf{a}=(a_1,a_2,\ldots,a_m)$ to $\sum_{i\in \sigmone}a_iv_i$. 
For $K\in\sigmk$ we define 
\[ 
\tilG_{K,\V}=\{\mathbf{a} \mid \eta(\mathbf{a})\in N\ \text{and}\ 
a_j=0\ \text{for $j\not\in K$}\}
\]
and define $G_{K,\V}$ to be the image of $\tilG_{K,\V}$ in 
$\tilT=\R^m/\Z^m$. 
It will be written $G_K$ for simplicity. 
The homomorphism $\eta$ restricted on $\tilG_{K,\V}$ induces 
an isomorphism 
\[\eta_K: G_K\cong H_{K,\V}\subset T=N_\R/N. \]

Put 
\[ 
G_{\Delta}=\bigcup_{I\in\sigmn}G_I\subset \tilT \quad \text{and}\quad 
  DG_{\Delta}=\bigcup_{I\in\sigmn}G_I\times G_I\subset 
   G_{\Delta}\times G_{\Delta}. 
\] 

Let $v(g)=\mathbf{a}=(a_1,a_2,\ldots,a_m)\in \R^m$ be 
a representative of $g\in \tilT$. The factor $a_i$ will be denoted by 
$v_i(g)$. It is determined modulo integers. If $g\in G_I$, then 
$v_i(g)$ is necessarily a rational number. Define a homomorphism 
$\chi_i:\tilde{T}\to \C^*$ by 
\[
 \chi_i(g)=e^{2\pi\img v_i(g)}
\]
 
Let $g\in G_I$ and $h=\eta_I(g)\in H_{I,\V}$. Then 
$\eta(v(g))\in N_{I}$ is a representative of $h$ in $N_I$ which 
will be denoted by $v(h)$. Then, for $g\in G_I$ and $i\in I$, 
\begin{equation*}\label{eq:vig}
v_i(g)\equiv\l \uiI, v(h)\r\quad \bmod \Z, 
\end{equation*}
and 
\[ 
\chi_i(g)=e^{2\pi\img \l\uiI, v(h)\r}=\chi_I(\uiI,h). 
\]

Let $\del$ be a complete simplicial multi-fan. Define 
\[ \T_T(\del,\V)=\sum_{g\in G_\del}\prod_{i\in\sigmone}
\frac{x_i}{1-\chi_i(g)e^{-x_i}}\in H_T^{**}(\del,\V)_\Q. \]
\begin{prop}\label{prop:ToddDH}
Let $\del$ be a complete simplicial multi-fan. Assume that 
$\xi\in H_T^2(\del,\V)$ is $T$-Cartier. Then 
\[ \pi_*(e^\xi\T_T(\del,\V)) =\sum_{u\in M}\DHF_{\PO(\xi_+)}(u)e^u. \]
Consequently 
\[ p_*(e^\xi\T_T(\del,\V))=\#(\POxi). \]
\end{prop}
\begin{proof}(cf. \cite{HM1} Section 8). 
Let $g\in G_\Delta$ and $I\in\sigmn$. If $g\notin G_I$, 
then there is an element $i\notin I$ such that $\chi_i(g)\not=1$; so 
\[
\frac{x_i}{1-\chi_i(g)e^{-x_i}}=
(1-\chi_i(g))^{-1}x_i+\text{higher degree terms}
\]
for such $i$ .
Hence $i_I^*(\frac{x_i}{1-\chi_i(g)e^{-x_i}})=0$.  Therefore, 
only elements $g$ in $G_I$ contribute 
to $\iota_I^*(\T_T(\Delta,\V))$.  Now suppose $g\in G_I$.  
Then $\chi_i(g)=1$ for $i\notin I$, so 
$\iota_I^*(\frac{x_i}{1-\chi_i(g)e^{-x_i}})=1$ 
for such $i$.  Finally, since $\iota_I^*(x_i)=\uiI$ for $i\in I$, 
we have  
\[
\iota_I^*(\T_T(\Delta,\V))=\sum_{g\in G_I}\prod_{i\in I}
\frac{\uiI}{1-\chi_i(g)e^{-\uiI}}.
\]
This together with \eqref{eq:DH3} shows that 
\begin{align*}
\pi_*(e^\xi\T_T(\Delta,\V))
 =&\pi_*\left(e^\xi\sum_{g\in G_\Delta}\prod_{i=1}^m\frac{x_i}
 {1-\chi_i(g)e^{-x_i}}\right) \\
 =&\sum_{I\in\Sigma^{(n)}}\frac{w(I)e^{\iota_I^*(\xi)}}{|H_{I,\V}|}
 \sum_{g\in G_I}\frac{1}{\prod_{i\in I}(1-\chi_i(g)e^{-\uiI})} \\
 =&\sum_{u\in M}\DHF_{\PO(\xi_+)}(u)e^u.
\end{align*}
\end{proof}

More generally, for $K\in \sigmk$, define $\T_T(\del,\V)_K$ by 
\[ \T_T(\del,\V)_K=\sum_{g\in G_{\del_K}}\prod_{i\in \sigm_K'^{(1)}}
\frac{x_i}{1-\chi_i(g)e^{-x_i}}\in H_T^{**}(\del,\V)_\Q. \]
Then the same proof as for Proposition \ref{prop:ToddDH} yields 
\begin{prop}\label{prop:ToddDH2}
Let $\del$ be a complete simplicial multi-fan. Assume that 
$\xi\in H_T^2(\del,\V)$ is $T$-Cartier. Then  
\[ \pi_*(e^\xi x_K\T_T(\del,\V)_K)
  =\sum_{u\in{A_K^*\cap M}}\DHF_{\PO_K(\xi_{K+})}(u)e^u. \]
for $K\in\sigmk$, where $x_K=\prod_{i\in K}x_i$. Consequently 
\[ p_*(e^\xi x_K\T_T(\del,\V)_K)=\#(\PO(\xi)_K). \]
\end{prop}

The lattice $M\cap A_K^*$ defines a volume element $dV_K$ on $A_K^*$. 
For $\xi=\sum_{i\in K\cup{\Sigma'_K}^{(1)}}d_ix_i\in H_T^2(\Delta_K,\V)_\R$, 
the volume $\vol\PO_{K}(\xi)$ of $\PO_{K}(\xi)$ is defined by 
\[ \vol\PO_{K}(\xi)=\int_{A_K^*}\DHF_{\PO_{K}(\xi)}dV_K^*. \]
\begin{prop}\label{prop:xivol} 
For $\xi=\sum_{i\in\sigmone}d_ix_i\in H_T^2(\Delta,\V)_\R$ 
\[ \frac{1}{|H_{K,\V}|}\vol\POxi_K=p_*(e^\xi x_K). \]
\end{prop}
\begin{proof}
We shall give a proof only for the case where $\xi$ is $T$-Cartier. 
The general case can be reduced to this case, cf. \cite{HM1}, Lemma 8.6. 
By Proposition \ref{prop:ToddDH2}
\[ \#(\POxi_K)=p_*(e^\xi x_K\T_T(\del,\V)_K). \]
The highest degree term with respect to $\{d_i\}$ in the right 
hand side is nothing but $\vol\POxi_K$ and is equal to 
\[ p_*\left(\frac{\xi^{n-k}}{(n-k)!}x_K\right)
\sum_{g\in G_{\Delta_K}}
  \left(\prod_{i\in \Sigma_K^{\prime(1)}}
  \frac{x_i}{1-\chi_i(g)e^{-x_i}}\right)_0, \]
where the suffix $0$ means taking $0$-th degree term.
But 
\[ \left(\prod_{i\in \Sigma_K^{\prime(1)}}
   \frac{x_i}{1-\chi_i(g)e^{-x_i}}\right)_0=
 \begin{cases}
 1, & g\in G_K \\
 0, & g\notin G_K.
 \end{cases} \]
Hence 
\[  \vol\POxi_K=|G_K|p_*\left(\frac{\xi^{n-k}}{(n-k)!}x_K\right)
=|H_{K,\V}|p_*(e^\xi x_K). \] 
\end{proof}

\section{Statement of main results}

Assume that $1\leq k$. 
For $J\in\sigmk$ let $M_J$ be the annihilator of $N_J$ and put
$\omega_J=u_1\wedge\ldots\wedge u_{n-k}\in 
\bigwedge^{n-k}M \subset\bigwedge^{n-k}M_\Q $ where 
$\{u_1,\ldots, u_{n-k}\}$ is an oriented basis of $M_J$. 
Define $f^J(x_i)\in\bigwedge^{n-k+1}M_\Q $ by 
\[
 f^J(x_i)=\iota_I^*(x_i)\wedge\omega_J\ \ 
\text{with $J\subset I\in \sigmn$}.
\] 
$f^J(x_i)$ is well-defined independently of $I$ containing $J$. 
Let $S^*(\bigwedge^{n-k+1}M_\Q)$
be the symmetric algebra over $\bigwedge^{n-k+1}M_\Q$.
$f^J:H_T^2(\del)_\Q\to \bigwedge^{n-k+1}M_\Q$ extends to 
$f^J:H_T^*(\del)_\Q\to S^*(\bigwedge^{n-k+1}M_\Q)$. 
For $x=\prod_ix_i^{\alpha_i}\in H_T^{2k}(\del)_\Q$ we put
\[ f^J(x)=(f^J(x_i))^{\alpha_i}. \]

The definition of $f^J$ depends on the orientations chosen, but 
$\frac{f^J(x)}{f^J(x_J)}$ does not. 
It belongs to the fraction field of 
the symmetric algebra $S^*(\bigwedge^{n-k+1}M_\Q)$ and 
has degree $0$. Hence it can be 
considered as an element of $Rat(\Proj(\bigwedge^{n-k+1}N_\Q))_0$, the 
field of rataional functions of degree $0$ on 
$\Proj(\bigwedge^{n-k+1}N_\Q)$. Let 
$\nu^*:Rat(\Proj(\bigwedge^{n-k+1}N_\Q))_0\to Rat(G_{n-k+1}(N_\Q))_0$
be the induced homomorphism of the Pl\"{u}cker embedding 
$\nu:G_{n-k+1}(N_\Q)\to\Proj(\bigwedge^{n-k+1}N_\Q)$. 
The image $\nu^*(\frac{f^J(x)}{f^J(x_J)})$ will be denoted 
by $\mu(x,J)$. 

Our first main result is stated in the following
\begin{theo}\label{theo:multimain}
Let $\del$ be a complete simplicial multi-fan and 
$x\in H_T^{2k}(\del,\V)_\Q$. 
For any $\xi\in H_T^2(\del)_\Q$ we have 
\begin{equation*}\label{eq:multimain} 
 p_*(e^\xi x)=\sum_{J\in \sigmk}\mu(x,J)
 p_*(e^\xi x_J)\ \ \text{in $Rat(G_{n-k+1}(N_\Q))_0$}.  
\end{equation*}
\end{theo}

\begin{coro}\label{coro:multiak}
Let $\del$ be a complete simplicial multi-fan in a lattice of 
rank $n$. Assume that $\xi\in H_T^2(\del,\V)$ is $T$-Cartier. Set 
\[ 
\#(\PO(\nu\xi))=\sum_{k=0}^n a_k(\xi)\nu^{n-k}. 
\]
Then we have 
\[
 a_k(\xi)=\sum_{J\in\sigmk}\mu_k(J)\vol \POxiJ 
\]
with
\[
 \mu_k(J)=\frac1{|H_{J,\V}|}\nu^*\left(\sum_{h\in H_{J,\V}}
\prod_{j\in J}\frac1{1-\chi(\ujJ,h)e^{-f^J(x_j)}}\right)_0 
\]
in $\Rat(G_{n-k+1}(N_\C))_0$. 
\end{coro}

\begin{note}
It can be proved without difficulty that $\mu_k(J)$ does not
depend on the choice of $\V$. Hence one has only consider the 
case where all the $v_i$  are primitive. 
\end{note}

For the following corollary we need to put an additional condition 
on the multi-fan $\del$. 
\begin{coro}\label{coro:multimain}
Let $\del$ be a multi-fan. Assume that there is a torus orbifold $X$ 
such that $\del$ is isomorphic to $\del_X$ and $H^*(X)_\Q$ 
is generated by $H^2(X)_\Q$. Then for 
$x\in H_T^{2k}(\del)_\Q$ the folloing equality holds. 
\[
 \barx=\sum_{J\in \sigmk}\mu(x,J)\barx_J
\ \ \text{in $Rat(G_{n-k+1}(N_\Q))_0$}\otimes_\Q H^{2k}(\del)_\Q, 
 \] 
where 
$\bar{x}$ is the image of $x\in H^*_T(\del)_\Q$ in $H^*(\del)_\Q$.
\end{coro}
\begin{rem}\label{rem:duality}
If $H^*(X)_\Q$ is generated by $H^2(X)_\Q$, then $H_T^*(\del_X)_\Q=
H_T^*(X)_\Q$. cf. \cite{Mas}, \cite{MP}. 
\end{rem}

\begin{rem}
When $\del$ is the fan associated to a convex lattice polytope $P$ and 
$\xi=D$, the Cartier divisor associated to $P$, 
we know (see, e.g. \cite{Ful}) that
\[
\mu_0(o)=1,\ a_0(\xi)=\vol\POxi,\ \mu_1(i)=\frac12, 
\ a_1(\xi)=\frac12\sum_{i\in\sigmone}\vol\POxi_i. 
\]
This is also true for simplicial multi-fans and $T$-Cartier $\xi$. 
\end{rem}
As to $a_n$ 
we have 
\[ a_n(\xi)=\Td[\del]. 
\]
In fact $a_n(\xi)=p_*(\T_T(\del,\V))=\left(\pi_*(\T_T(\del,\V))\right)_0$. 
Thus the above equality follows from the following rigidity property: 
\begin{theo}
Let $\del$ be a complete simplicial multi-fan. Then 
\[
\pi_*(\T_T(\del,\V))=\left(\pi_*(\T_T(\del,\V))\right)_0=\Td[\del]. 
\]
\end{theo}
See \cite{HM1} Theorem 7.2 and its proof. Note that $\Td[\del]=1$ 
for any complete simplicial fan $\del$. 

The explicit formula for $\pi_*(\T_T(\del,\V))$ is given by 
\[
\pi_*(\T_T(\del,\V))=\sum_{I\in\sigmn}\frac{w(I)}{|H_{I,\V}|}
\sum_{h\in H_{I,\V}}\prod_{i\in I}\frac1{1-\chi_I(\uiI, h)e^{-\uiI}}. 
\]
This does not depend on the choice of $\V$ and is in fact equal 
to $\Td[\del]$. 

Let $\del$ be a (not necessarily complete) simplicial fan 
in a lattice of rank $n$. Set 
\[
Td_T(\del)=\sum_{I\in \sigmn}\frac1{|H_I|}\sum_{h\in H_I}\prod_{i\in I}
\frac1{1-\chi_I(\uiI, h)e^{-\uiI}}\in S^{-1}H_T^{**}(pt)_\Q. 
\]
For a simplex $I$ let $\sigm(I)$ be the simplicial complex consisting 
of all faces of $I$. 
For a fan $\del(I)=(\sigm(I),C)$, $Td_T(\del(I))$ is denoted by $Td_T(I)$. 
\begin{theo}
$Td_T(I)$ is additive with respect to simplicial subdivisions of  
the cone $C(I)$. Namely, if $\del$ is the fan determined 
by a simplicial subdivison of $C(I)$, then the following equality holds 
\begin{equation*}\label{eq:addTd}
 Td_T(\del)=Td_T(I). 
\end{equation*}
\end{theo}
For the proof it is sufficient to assume that $\del(I)$ and $\del$ are 
non-singular. In such a form a proof is given in \cite{Mor}. 
The following corollary ensures that $\mu_k(J)$ can be defined 
for general polyhedral cones as pointed out by Morelli in \cite{Mor}. 
\begin{coro}
Let $\del(J)=(\sigm(J),C)$ be a fan in a lattice $N$ of rank $n$ 
where $J$ is a simplex of dimension $k-1$. Then 
$\mu_k(J)\in \Rat(G_{n-k+1}(N_\Q))_0$ is additive with respect to 
simplicial subdivisions of $C(J)$. 
\end{coro}

\section{Proof of Theorem \ref{theo:multimain} and Corollary \ref{coro:multimain}}
Proof of Theorem \ref{theo:multimain}. 
For a primitive sublattice $E$ of $N$ of rank $n-k+1$ let 
$w_E\in\bigwedge^{n-k+1}N$ be a representative of 
$\nu(E)\in \Proj(\bigwedge^{n-k+1}N_\Q)$. The equality in 
Theorem \ref{theo:multimain}
is equivalent to the condition that 
\[ p_*(e^\xi x)=\sum_{J\in \sigmk}\frac{f^J(x)}{f^J(x_J)}(w_E)p_*(e^\xi x_J)
 \ \ \text{holds for every generic $E$}.  \]

Let $E$ be a generic primitive sublattice in $N$ of rank $n-k+1$. 
The intersection $E\cap N_J$ has rank one for each $J\in\sigmk$. 
Take a non-zero vector $v_{E,J}$ in $E\cap N_J$. 
(One can choose $v_{E,J}$ to be the unique primitive vector contained in 
$E\cap C(J)$. But any non-zero vector 
will suffice for the later use.)
For $x\in H_T^{2k}(\Delta)$ and $J\in \sigmk$ the value of 
$\iota_I^*(x)$ evaluated on $v_{E,J}$ for $I\in \sigmn$ containing $J$ 
depends only on $\iota_J^*(x)$ so that it will be denoted by 
$\iota_J^*(x)(v_{E,J})$. Similarly 
we shall simply write $\l u_j^J, v_{E,J}\r$ 
instead of $\l u_j^I, v_{E,J}\r$. 
\begin{lemm}\label{lemm:wEvJ}
Put $f_j^J=u_j^J\wedge \omega_J$. Then 
\[ a\l f_j^J, w_E\r =\l u_j^J, v_{E,J}\r, \]
where $a$ is a non-zero constant depending only on $v_{E,J}$. 
\end{lemm}

\begin{proof}
Take an oriented basis $u_1,\ldots,u_{n-k}$ of $M_J$. Take also 
a basis $w_1,\ldots,w_{n-k+1}$ of $E$ and write $v_{E,J}
=\sum_lc_lw_l$. Then, since $\l u_i,v_{E,J}\r=0$, 
\[ \sum_{l=1}^{n-k+1}c_l\l u_i, w_l\r =0, \quad 
 \text{for $i=1,\ldots,n-k$}. \] 
The matrix $\left(a_{il}\right)=\left(\l u_i,w_l\r\right)$ has rank $n-k$ and
we get 
\[ (c_1,\ldots, c_{n-k+1})=a(A_1,\ldots, A_{n-k+1}),\quad a\not=0, \]
where
\[ A_l=(-1)^{l-1}\det\begin{pmatrix}
 a_{11} &\dots &\widehat{a_{1l}} &\dots &a_{1\, n-k+1} \\
 \hdotsfor{5} \\
 \hdotsfor{5} \\
 a_{n-k\, 1} &\dots &\widehat{a_{n-k\, l}} &\dots &a_{n-k\, n-k+1} 
 \end{pmatrix}. \]
Then 
\[ \begin{split}
 \l u_j^J,v_{E,J}\r
    &=\sum_{l=1}^{n-k+1}c_l\l u_j^J,w_l\r \\
    &=a\sum_{l=1}^{n-k+1}\l u_j^J,w_l\r A_l \\
    &=a\det\begin{pmatrix}
     \l u_j^J,w_1\r &\dots &\l u_j^J,w_{n-k+1}\r \\
     \l u_1,w_1\r   &\dots &\l u_1,w_{n-k+1}\r \\
     \hdotsfor{3} \\
     \l u_{n-k},w_1\r   &\dots &\l u_{n-k},w_{n-k+1}\r 
     \end{pmatrix} \\
    &=a\l f_j^J, w_E \r
   \end{split} \] 
where $f_j^J=u_j^J\wedge u_1\cdots\wedge u_{n-k}$ and 
$w_E=w_1\wedge\cdots\wedge w_{n-k+1}$.     
\end{proof}

\begin{rem}
Let $X$ be a torus orbifold of dimension $2n$ 
and $\Delta$ the associated multifan.  
Let $T=T^n$ be the compact torus acting on $X$. $E\cap N_J$ determines 
a subcircle $T^1_{E,J}$ of $T$. 
Then $T^1_{E,J}$ pointwise fixes an invariant complex suborbifold $X_J$. 
Some of its covering acts on the normal vector space of an Euclidean covering 
of an invariant neighborhood at each generic point in 
$X_J$. Then the numbers $\l u_j^J, v_{E,J}\r$ are weights of this 
action. 
\end{rem}

Lemma \ref{lemm:wEvJ} implies that 
\[ \frac{f^J(x)}{f^J(x_J)}(w_E)=\frac{\iota_J^*(x)}
 {\prod_{j\in J}u_j^J}(v_{E,J}). \]
Then the equality in Theorem \ref{theo:multimain} holds if and only if  
\begin{equation}\label{eq:1prime} 
 p_*(e^\xi x)=\sum_{J\in \sigmk} \frac{\iota_J^*(x)}
 {\prod_{j\in J}u_j^J}(v_{E,J})p_*(e^\xi x_J) 
\end{equation} 
holds for every generic $E$. 

The following lemma is easy to prove, cf. e.g. \cite{HM1} Lemma 8.1. 
\begin{lemm}\label{lemm:span}
The vector space $H_T^{2k}(\Delta)_\Q$ is spanned by elements of 
the form
\[ u_1\cdots u_{k_1}x_{J_{k_1}},\ J_{k_1}\in 
 \Sigma^{(k-k_{1})}, \ u_i\in M_\Q, \]
with $0\leq k_1\leq k-1$. 
\end{lemm} 

\begin{note} 
For $x=u_1\cdots u_{k_1}x_{J_{k_1}},\ J_{k_1}\in 
\Sigma^{(k-k_{1})}$, with $k_1\geq 1$, 
\[ p_*(e^\xi x)=0. \]
\end{note}
In view of this lemma we proceed by induction on $k_1$ for 
$x=u_1\cdots u_{k_1}x_{J_{k_1}}$. 
 
For $x=x_{J_0}$ with $J_0\in \sigmk$, the left hand side of 
\eqref{eq:1prime} 
is equal to $p_*(e^\xi x_{J_0})$. Since $i_J^*(x)=0$ unless 
$J=J_0$ and $i_J^*(x)/\prod_{j\in J}u_j^J=1$ for $J=J_0$, the 
right hand side 
is also equal to $p_*(e^\xi x_{J_0})$. Hence \eqref{eq:1prime} 
holds with $x$ of the form 
$ x=x_{J_0} \ \text{for $J_0\in \sigmk$}$. 

Assuming that \eqref{eq:1prime} holds for $x$ of the form 
$u_1\cdots u_{k_1}x_{J_{k_1}}$ with 
$J_{k_1}\in\sigm^{(k-k_1)}$, we shall prove that 
it also holds for $x=u_1\cdots u_{k_1}u_{k_1+1}x_{J_{k_1+1}}$
with $J_{k_1+1}\in \sigm^{(k-(k_1+1))}$. Put $K=J_{k_1+1}$. 
\vspace*{0.2cm}\\
Case a). $u_{k_1+1}$ belongs to $M_{K\Q}$, 
that is, $\l u_{k_1+1},v_i\r=0$ for all $i\in K$. In this 
case 
\[ u_{k_1+1}=\sum_{i\in\sigmone\setminus K}\l u_{k_1+1},v_i\r x_i \]
since $\l u_{k_1+1},v_i\r=0$ for all $i\in K$. For $i\not\in K$, 
$x_ix_{J_{k_1+1}}$ is either of the form $x_{J^i}$ 
with $J^i\in\sigm^{(k-k_1)}$ or equal to $0$. Thus, for 
$x=u_1\cdots u_{k_1}x_ix_{J_{k_1+1}}$ with $i\not\in K$, 
the equality \eqref{eq:1prime} 
holds by induction assumption, and it also holds for 
$x=u_1\cdots u_{k_1}u_{k_1+1}x_{J_{k_1+1}}$ by linearity. 
\vspace*{0.2cm}\\
Case b). General case.
We need the following
\begin{lemm}\label{lemm:MKQ} 
For $K\in \sigm^{(k-k_1)}$ with $k_1\geq 1$, the composition homomorphism 
$M_{K\Q}\subset M_\Q\to E_\Q^*$ is surjective.
\end{lemm}
The proof will be given later. By this lemma, there exists an 
element $u\in M_{K\Q}$ such that
\[ \l u_{k_1+1},v_{E,J}\r=\l u,v_{E,J}\r \ \text{for all $J\in \sigmk$}. \]
Note that $\l \iota_J^*(u),v_{E,J}\r=\l u,v_{E,J}\r$ for any 
$u\in M_\Q$. 
Then, in \eqref{eq:1prime} for $x=u_1\cdots u_{k_1}u_{k_1+1}x_{J_{k_1+1}}$
with $J_{k_1+1}\in \sigm^{(k-(k_1+1))}$, we have 
\[ \begin{split}\iota_J^*(x)(v_{E,J})
  &=(\prod_{i=1}^{k_1}\l u_i,v_{E,J}\r)\l u_{k_1+1},v_{E,J}\r \\
  &=(\prod_{i=1}^{k_1}\l u_i,v_{E,J}\r)\l u,v_{E,J}\r.
  \end{split} \]
Hence if we put $x'=u_1\cdots u_{k_1}ux_{J_{k_1+1}}$, the right 
hand side of \eqref{eq:1prime} is equal to 
\[ \sum_{J\in \sigmk} \frac{\iota_J^*(x')}
 {\prod_{j\in J}u_j^J}(v_{E,J})p_*(e^\xi x_J). \]
This last expression is equal to $p_*(e^\xi x')$ since 
$x'$ belongs to Case a). Furthermore 
$p_*(e^\xi x')=0$ and $p_*(e^\xi x)=0$ by Note after 
Lemma \ref{lemm:span}. Thus both side of 
\eqref{eq:1prime} for $x=u_1\cdots u_{k_1}u_{k_1+1}x_{J_{k_1+1}}$ 
are equal to $0$. This completes the proof of Theorem 
except for the proof of Lemma \ref{lemm:MKQ}. 
\vspace*{0.2cm}\\
\noindent Proof of Lemma \ref{lemm:MKQ}. 

Take a simplex $I\in \sigmn$ which contains $K$ and a simplex 
$K'\in \sigm^{(k-1)}$ such that $K\subset K'\subset I$. 
Such a $K'$ exists since $k-k_1\leq k-1$. Then there are 
exactly $n-k+1$ simplices $J^1,\ldots,J^{n-k+1}\in\sigmk$ such that 
$K'\subset J^i\subset I$. It is easy to see that the 
vectors $v_{E,J^1},\ldots,v_{E,J^{n-k+1}}$ are linearly independent
so that they span $E_\Q$. Moreover $M_{K'\Q}$ detects these 
vectors, that is, $M_{K'\Q}\to M_\Q\to E^*_\Q$ is surjective. Since 
$M_K'\subset M_K\subset M$, $M_{K\Q}\to E^*_\Q$ is surjective. 
\vspace*{0.2cm}\\
\noindent Proof of Corollay \ref{coro:multiak}

By Proposition \ref{prop:ToddDH2}
\[
 \#(\PO(\nu\xi))=p_*(e^{\nu\xi}\T_T(\del,\V))=\sum_{k=0}^n a_k(\xi)\nu^{n-k}. 
\]
Put $x=\left(\T_T(\del,\V)\right)_k\in H_T^{2k}(\del,\V)_\Q$. 
By Theorem \ref{theo:multimain} and Proposition \ref{prop:xivol}

\[
 a_k(\xi)=\sum_{J\in\sigmk}\nu^*\left(\frac{f^J(x)}{f^J(x_J)}\right)
\frac{\vol\POxi_J}{|H_{J,\V}|}.
 \] 
Thus it suffices to show that
\[ \frac{f^J(x)}{f^J(x_J)}=\left(\sum_{h\in H_{J,\V}}
\prod_{j\in J}\frac1{1-\chi(\ujJ,h)e^{-f^J(x_j)}}\right)_0, \]
or 
\[ f^J(x)=\left(\sum_{h\in H_{J,\V}}
\prod_{j\in J}\frac{f^J(x_j)}{1-\chi(\ujJ,h)e^{-f^J(x_j)}}\right)_k. \]

Let $g\in G_\del$. If $g\not\in G_J$, then there is an element  
$i\not\in J$ such that $\chi_i(g)\not=1$, and, for such $i$, 
\[ f^J(\frac{x_i}{1-\chi_i(g)e^{-x_i}})=
f^J((1-\chi_i(g))^{-1}x_i+ \text{higher degree terms})=0, \]
since $f^J(x_i)=0$. Thus 
\[ 
f^J\left(\prod_{i\in\sigmone}\frac{x_i}{1-\chi_i(g)e^{-x_i}}\right)=0
\]
for $g\not\in G_J$. 

If $g\in G_J$, then $\chi_i(g)=1$ for $i\not\in J$. Thus 
\[
f^J\left(\frac{x_i}{1-\chi_i(g)e^{-x_i}}\right)
=f^J(1+\frac12x_i+\text{higher degree terms})=1
\]
for $g\in G_J,\ i\not\in J$. It follows that 
\[
f^J\left(\sum_{g\in G_\del}\prod_{i\in\sigmone}
\frac{x_i}{1-\chi_i(g)e^{-x_i}}\right)=
\sum_{g\in G_J}\prod_{i\in J}
\frac{f^J(x_i)}{1-\chi_i(g)e^{-f^J(x_i)}}.
\]
This implies 
\[ f^J(\T_T(\del,\V)_k)=\left(\sum_{h\in H_{J,\V}}\prod_{j\in J}
\frac{f^J(x_j)}{1-\chi_J(u_j^J,h)e^{-f^J(x_j)}}\right)_k.
\]
This finishes the proof of Corollary \ref{coro:multiak}.
\vspace*{0.3cm}\\
\noindent Proof of Corollary \ref{coro:multimain}. 

Put $x'=\sum_{J\in \sigmk}\mu(x,J)x_J$. Then
\[ p_*(e^\xi x')
  =\sum_{J\in \sigmk}\mu(x,J)p_*(e^\xi x_J) 
  =p_*(e^\xi x)\] 
by Theorem \ref{theo:multimain}. It follows that $p_*(e^\xi(x'-x))=0$. 
Thus, in order to prove Corollary \ref{coro:multimain}, it suffices to show 
that $p_*(e^\xi y)=0,\ \forall \xi\in H_T^2(\del)_\Q$, implies 
that $p^*(y)=0$ . 
By the assumption $\del$ is isomorphic to $\del_X$ where $X$ is a 
torus orbifold such that $H^*(X)_\Q$ is generated by $H^2(X)_\Q$. 
For such $X$ we know that $H_T^*(\del)_\Q=H_T^*(X)_\Q$ and 
$H^*(\del)_\Q=H^*(X)_\Q$ by Remark \ref{rem:duality}.   
In particular $H^*(\del)_\Q$ satisfies the Poincar\'{e} duality. 
It follows that $p_*(e^\xi y)=0$ 
for all $\xi$ implies that $p_*(y)=0$.

\providecommand{\bysame}{\leavevmode\hbox to3em{\hrulefill}\thinspace}

\end{document}